\documentclass[11 pt,twoside]{amsart} 
\usepackage{hyperref,pifont,amssymb,amsmath,amsthm,lipsum,placeins,amsfonts,amsopn, latexsym,graphicx, mathrsfs, verbatim, enumerate, url, multicol,color,footnote,subfigure,float}
\usepackage[stable]{footmisc}
\usepackage[margin=1in]{geometry} 
\usepackage[stable]{footmisc} 

\def\address#1{\expandafter\def\expandafter\@aabuffer\expandafter
	{\@aabuffer{\affiliationfont{#1}}\relax\par
	\vspace*{13pt}}}

\DeclareMathOperator{\re}{Re}

\DeclareMathOperator{\BO}{O}

\begin{document}

\numberwithin{equation}{section}
\newtheorem{thm}{Theorem}
\newtheorem*{thm*}{Theorem}
\newtheorem*{thmA}{Theorem A}
\newtheorem*{thmf}{Theorem 1}
\newtheorem*{thmg}{Theorem 2}
\newtheorem*{thmh}{Theorem 3}
\newtheorem*{thmB}{Theorem A \cite{L3}}
\newtheorem*{propC}{Proposition C}
\newtheorem{thmx}{Theorem}
\newtheorem{conj}{Conjecture}
\newtheorem{sublem}{Lemma}[thm]
\newtheorem{lem}[thm]{Lemma} 
\newtheorem{cor}[thm]{Corollary}
\newtheorem{prop}[thm]{Proposition}
\theoremstyle{remark}
\newtheorem{rem}[thm]{Remark}
\newtheorem*{rem*}{Remark}
\newtheorem{remx}[thmx]{Remark}
\newtheorem{ex}[thm]{Example}
\theoremstyle{definition}
\newtheorem{Def}[thm]{Definition}
\newtheorem*{Q}{Question}
\newtheorem*{notation}{Notation}
\renewenvironment{proof}[1][\proofname]{{\bfseries #1. }}{\qed \\ \newline} 

\newcommand{\lo}{o}
\newcommand{\0}{\mbox{O}}
\newcommand{\ord}{\mbox{ord}}
\newcommand{\R}{\mathbb{R} }
\newcommand{\C}{\mathbb{C} }
\newcommand{\Z}{\mathbb{Z}}
\newcommand{\p}{\mathbb{P}}
\newcommand{\ol}{\overline}
\newcommand{\N}{\mathbb{N}}
\newcommand{\wt}{\widetilde}
\newcommand{\Id}{\operatorname{Id}}
\newcommand{\Arg}{\operatorname{Arg}}
\newcommand{\A}{\hat{A}_2}
\newcommand{\B}{\hat{B}_2}
\newcommand{\F}{\mathcal{F}}
\newcommand{\Oa}{\Omega_1}
\newcommand{\Ob}{\Omega_2}
\newcommand{\Oc}{\Omega_3}
\newcommand{\bound}{\mathcal{B}}

\newcommand\blfootnote[1]{
  \begingroup
  \renewcommand\thefootnote{}\footnote{#1}
  \addtocounter{footnote}{-1}
  \endgroup
}
\renewcommand{\thefootnote}{\alph{footnote}}

\title{Interesting examples in $\C^2$ of maps tangent to the identity without domains of attraction}

\author{Sara Lapan}
\subjclass[2010]{Primary: 37F10; Secondary: 32H50}  
\maketitle
\begin{center}Department of Mathematics\\  University of California, Riverside \\ Riverside, CA; USA \\
slapan@math.ucr.edu
\end{center}
\begin{abstract}We give an interesting example of a map in $\C^2$ that is tangent to the identity, but that does not have a domain of attraction along any of its characteristic direction.  This map has three characteristic directions, two of which are not attracting while the third attracts points to that direction, but not to the origin.  In addition, we show that if we add higher degree terms to this map, sometimes a domain of attraction along one of its characteristic directions will exist and sometimes one will not.\end{abstract}

\pagestyle{myheadings}
\markright{Interesting Examples in $\C^2$}

\section{Introduction}
In this article, we give an example of a holomorphic map tangent to the identity in $\C^2$ that exhibits interesting dynamical behavior and a couple modifications to that example that demonstrate different dynamical behavior. \\

A germ of a holomorphic self-map of $\C^m$, $F$, is \textit{tangent to the identity at $p$} if for $p\in\C^m$, $F(p)=p$, $\mathrm{d}F(p)=\Id$, and $F\not\equiv\Id$. We can move $p$ to the origin via a conjugation of $F$ by an affine map, so we assume $p$ is the origin. Near the origin, $F$ can be written as:
$$F(z):=z+P_{k+1}(z)+P_{k+2}(z)+\cdots,$$
where each $P_j$ 
is a homogeneous polynomial map of degree $j$ and $k+1$ is the \textit{order} of $f$ ($P_{k+1}\not\equiv 0$). For $v\in\C^m\setminus\{O\}$, $[v]\in\p^{m-1}(\C)$ is a \textit{characteristic direction} of $F$ if $P_{k+1}(v)=\lambda v$ for some $\lambda\in\C$ and $[v]$ is \textit{degenerate} if $\lambda=0$, otherwise it is \textit{non-degenerate}.  For $F$ with characteristic direction $[v]$, $\Omega\subset\C^m$ is a \textit{domain of attraction to the origin along $[v]$} if: $\Omega$ is a domain with the origin in its boundary, $f(\Omega)\subset\Omega$, and, for any $z\in\Omega$, $z$ converges to the origin along $[v]$ (i.e., $f^n(z)\to\0$ and $[f^n(z)]\to[v]$). \\

Write $P_{k+1}:=(p,q):\C^{m}\to\C\times\C^{m-1}$ and $z:=(x,y)\in\C\times\C^{m-1}$. Let $r(x,y):=xq(x,y)-yp(x,y)$. Then $[v]=[x_0:y_0]\in\p^{m-1}(\C)$ is a characteristic direction exactly when $r(x_0,y_0)=\0$. A non-degenerate characteristic direction $[v]$ of $F$ has constants associated to it, called \textit{directors} (also called Hakim's index), which are the eigenvalues of the linear operator $\frac{1}{k}\left(\mathrm{d}(P_{k+1})_{[v]}-\Id\right):T_{[v]}\p^{m-1}\to T_{[v]}\p^{m-1}$; equivalently, the directors of $[v]=[1:y_0]$ are the eigenvalues of $A(v):=\frac{1}{k}\left(\frac{\mathrm{d}r(1,y)}{p(1,y)}\right)\Big|_{y=y_0}$ (see \cite{AR}). \\

We shall further analyze characteristic directions using a classification given by Abate and Tovena in \cite{AT} and discussed in \cite{RV}. Assume dimension $m=2$ and $[v]=[1:y_0]$ is a characteristic direction of $f$. Let $\mu_1(y_0), \mu_2(y_0)\in\N$ be the orders of vanishing of $p(1,y),r(1,y)$, respectively, at $y=y_0$. Then the direction $[1:y_0]$ is called:
\begin{itemize}
\item an \textit{apparent} characteristic direction if $\mu_2(y_0)<\mu_1(y_0)+1$;
\item a \textit{Fuchsian} characteristic direction if $\mu_2(y_0)=\mu_1(y_0)+1$;
\item an \textit{irregular} characteristic direction if $\mu_2(y_0)>\mu_1(y_0)+1$.
\end{itemize}
If $[1:y_0]$ is a characteristic direction, then $r(1,y_0)=0$, so $\mu_2(y_0)\geq1$. A direction is non-degenerate when $p(1,y_0)\neq0$ or, equivalently, $\mu_1(1,y_0)=0$, so non-degenerate characteristic directions are either Fuchsian or irregular. \\  

In our main example, we demonstrate that there is no domain of attraction to the origin along any of its characteristic directions.  In addition, we show that one of the directions attracts points to itself, but not to the origin.  Consequently, within any small open neighborhood of the origin, there are points (an open subset of $\R^2$) that are not fixed and that, under iteration, remain close to the origin but do not converge to it.  This map, as far as the author knows, is the first example of a map tangent to the identity that is shown to exhibit these interesting behaviors. 
The additional examples we discuss are modifications of this main example and are created by adding higher degree terms to the main example.  We show that this addition of higher degree terms can lead to interesting changes in the dynamics near the origin. \\

The following theorem is about our main example. We denote the $n$-th iterate of $f$ acting on a point $(z,w)\in\C$ as $(z_n,w_n):=f^n(z,w)$ and $(z_0,w_0):=(z,w)$. 

\begin{thm}\label{mainthm}
Let $f(z,w)=\left(z\left(1-(z-w)\right),w\left(1+(z-w)\right)\right)$  Then $f$ has the following properties:
\begin{enumerate}
\item at the origin, $f$ has characteristic directions $[1:0],[0:1],[1:1]$;
\item the complex line $\{z=w\}$, which corresponds to $[1:1]$, is fixed pointwise;
\item there is a domain $A\subset\R^2\subset\C^2$ with the origin in its boundary whose points converge to the real line $\{z=w\}$, but do not converge to the origin; 
\item for $(z,w)\in A$, $\displaystyle\lim_{n\to\infty} n|z_n-w_n|<1$; 
\item there is no domain of attraction to the origin;
\item at $(1,1)$, $f^2$ is tangent to the identity with 3 characteristic directions: one is a degenerate apparent direction and two are non-degenerate directions with negative directors, hence there is no domain of attraction along those two directions; and
\item for every point $x\in I_1\cup I_2$, its preimage set in $\R^2$, $\cup_{n>0}f^{-n}(x)\cap\R^2$, has a limit point on $\{z=w\}$ distinct from the origin.
\end{enumerate}
\end{thm}

The convergence of $z_n-w_n$ to $0$ might be significantly faster than the bound $n^{-1}$ given in \textit{(4)}. Section \ref{R2} contains a proof of \textit{(7)}. The line segments $I_1:=[0,1]\times\{0\}$ and $I_2:=\{0\}\times[0,1]$ are attracted to the origin by $f$ along $[1:0]$ and $[0:1]$, respectively. Note that $f$ as in Theorem \ref{mainthm} appears as $(3_{110})$ in Abate's classification of quadratic maps tangent to the identity in $\C^2$ in \cite{A2}.  \\ 

\begin{rem*}In this article, it is sometimes more convenient to perform a linear conjugation of $f$ that sends $[1:1]$ to $[1:0]$.  This linear conjugation is discussed in \eqref{xy}.  The new expression for $f$ is:
\begin{equation}\label{ftilde}\tilde{f}(x,y)=\left(x-y^2,y-xy\right)\qquad\mbox{ with }[1:0]\mbox{ degenerate characteristic direction.}\end{equation}\end{rem*}

We find this way of expressing $f$ particularly useful in \S\ref{hot}, where we add higher degree terms and study when a domain of attraction exists.  In \S\ref{hot}, we show that some choices of higher degree terms can cause a domain of attraction to the origin to exist while other choices cause one not to exist.  We use~\cite[Theorem A]{L3}, which can be found in \S\ref{hot}, to see that there are many choices of higher degree terms that can lead to a domain of attraction along $[1:0]$, but not along its other characteristic directions.  The following theorem is a (significant) simplification of \cite[Theorem A]{L3} to maps whose lower degree terms are of the form \eqref{ftilde}.

\begin{thm}\label{thmg} 
Let $g(x,y)=(x-y^2+ax^{r+1},y-xy)$, where: (1) $a\in\C^\times$ and $r\in\N,r\geq3$; or (2) $a\not\in\R_{\geq0}$ and $r=2$.  Then $g$ has a domain of attraction to the origin along $[1:0]$.\end{thm}

\begin{notation} Let $||\cdot||$ denote the standard Euclidean norm. We use big-O notation so $h=\BO(g_1,\ldots,g_s)$, for $h,g_1,\ldots,g_s:\C^m\to\C^n$, means $\exists C_1,\ldots,C_s>0$ such that $||h(z)||\leq C_1||g_1(z)||+\cdots+C_s||g_s(z)||$. For $t\in\N$, $h=\BO((g_1,g_2)^t)$ means $h=\BO(g_1^t,g_1^{t-1}g_2,\ldots,g_2^t)$.
\end{notation}

Using this notation and \cite[Theorem A]{L3}, the previous theorem also holds for $g+\hat{g}$, where $\hat{g}(x,y)=\left(y\BO((x,y)^2)+\BO(x^{r+2}),~y\BO((x,y)^2)+\BO(x^{r+2})\right).$\\ 

In addition, we show that some choices of higher degree terms cause there to be no domain of attraction to the origin along any characteristic direction.  

\begin{thm}\label{thmh}
Let $h(x,y)=(x-y^2+ax^3,y-xy)$.  If $a\in\R_{>0}$, then $h$ has no domain of attraction to the origin along any direction.\end{thm}


\textbf{Acknowledgements.} The author would like to thank Laura DeMarco for useful comments on an earlier draft of this paper.  The author would like to thank Romain Dujardin for asking the author a question that inspired this paper. The author would also like to thank the referee for useful comments and remarks that improved the presentation of this paper.

\section{The Main Example}\label{s1}
Let $f:\C^2\to\C^2$ be the map:
$$f(z,w)=\left(z+p(z,w),w+q(z,w)\right):=\left(z(1-(z-w)),w(1+(z-w))\right).$$
Then $f$ is tangent to the identity at the origin 
and its characteristic directions can be found by finding the zeros of:
\begin{equation}\label{r}r(z,w)=zq(z,w)-wp(z,w)=2zw(z-w).\end{equation}
In particular, $f$ has exactly three characteristic directions: $[1:0],[0:1],$ and $[1:1]$.  The directions $[1:0]$ and $[0:1]$ are both non-degenerate because $(p,q)(1,0)\neq (0,0)$ and $(p,q)(0,1)\neq(0,0)$.  Since both of these directions are non-degenerate, they have corresponding directors.  \\ 


The director of $[1:0]$ is:  
\begin{equation}\label{dir1}\left(\frac{1}{p(1,w)}\frac{d}{dw} r(1,w)\right)\Bigg|_{w=0}=-2<0.\end{equation}
To find the director of $[0:1]$ we switch the roles of coordinates $z$ and $w$. The director of $[0:1]$ is:
\begin{equation}\label{dir2}\left(\frac{1}{q(z,1)}\frac{d}{dz} (-r(z,1))\right)\Bigg|_{z=0}=-2<0.\end{equation}

Since the directors of $[1:0]$ and $[0:1]$ are both negative, by \cite[Corollary 8.11]{AR} there is no domain of attraction whose points converge to the origin along $[1:0]$ or $[0:1]$.\\ 

If the orbit of a point converges to the origin along a direction, that direction must be a characteristic direction by \cite[Proposition 1.3]{H1}.  Hence, the only directions along which points could converge to the origin are $[1:0],[0:1],$ and $[1:1]$.  We just showed that there is no domain of attraction along $[1:0]$ or $[0:1]$, however some points are attracted to the origin along those directions.  In particular, on $\{w=0\}$:
$$f(z,0)=(z(1-z),0)$$
attracts the cauliflower set $\mathcal{C}\times\{0\}\subset\C\times\{0\}$ to the origin along $[1:0]$, where:
\begin{align*}\displaystyle\mathcal{C}:=&\mbox{interior of }\{z\in\C \ | \ F(z):=z(1-z), \lim_{n\to\infty}|F^n(z)|\not\to\infty\}\\
=& \mbox{the domain of attraction to the origin for $F(z)$}
\end{align*}
  Similarly, on $\{z=0\}$, 
$$f(0,w)=(0,w(1-w)),$$
attracts $\{0\}\times\mathcal{C}$ to the origin along $[0:1]$.  Hence, under iteration by $f$, the sets $f^{-k}\left(\mathcal{C}\times\{0\}\right)$ and $f^{-k}\left(\{0\}\times\mathcal{C}\right)$ converge to the origin along $[1:0]$ and $[0:1]$, respectively, for all $k\in\N$.  So there are complex curves, but not domains, that are attracted to the origin along those two characteristic directions.  In \S\ref{11}, we analyze the dynamics of $f$ along the remaining characteristic direction $[1:1]$. \\

The complex line $z=w$ is fixed pointwise so we are interested in the types of fixed points that occur on this line. The eigenvalues of $\mathrm{d} f(z,z)$ are $1$ and $1-2z$, with corresponding eigenvectors $(1,1)$ and $(1,-1)$, respectively, when $z\neq0$. The fixed point $(z,z)$ is:
\begin{itemize}
\item semi-attracting when $\left|z-\frac{1}{2}\right|<\frac{1}{2}$ (equivalently, $|1-2z|<1$); 
\item parabolic when $\left|z-\frac{1}{2}\right|=1$; and 
\item semi-repelling when $\left|z-\frac{1}{2}\right|>1$.
\end{itemize}
On $\R^2$, this translates to: semi-attracting when $0<z<1$, parabolic when $z=0$ or $z=1$, and semi-repelling when $z<0$ or $z>1$.

\section{Illustrating the dynamics of $f$ in $\R^2$}\label{R2}
To more easily visualize $f$, for this section we restrict the domain of $f$ to $\R^2$.  
Let $I_1=[0,1]\times\{0\}$ and $I_2=\{0\}\times[0,1]$ since $\mathcal{C}\cap\R=(0,1)$.  The endpoints of $I_1$ and $I_2$ (and their preimages) are preimages of the origin while the rest of $I_1$ and $I_2$ are attracted to the origin along $[1:0]$ and $[0:1]$, respectively.  The curve segments given by $I_1,I_2$ and their preimages are depicted in Figure \ref{fig1} and are the boundary of the blue bounded region in Figure \ref{fig2}. \\

Figure \ref{fig2} suggests that the preimages of $I_1$ and $I_2$ limit towards an intersection with $\{z=w\}$. By following the preimages of $\{z=0\}$ and $\{w=0\}$ in more detail, we show this is the case. To simplify this discussion, we follow the preimages that are distinct from $\{z=0\}$ and $\{w=0\}$ since they are contained in their own preimages: $f^{-1}(\{z=0\})=\{z=0\}\cup\{z-w=1\}$ and $f^{-1}(\{w=0\})=\{w=0\}\cup\{z-w=-1\}$. In particular, we follow the preimages of $l^+$ and $l^-$, 
which are defined so that $f^{-1}(I_1)=I_1\sqcup l^-$ and $f^{-1}(I_2)=I_2\sqcup l^+$. 
This implies that:
\begin{align}
l^+&=\{z-w=1 \ | \ 0\leq w\leq \frac{1}{2}\} \quad&\quad l^-=&\{z-w=-1 \ | \ 0\leq z\leq \frac{1}{2}\}\label{lpm}
\end{align} 
For $n\geq0$:
\begin{align*}
z_{n+1}-w_{n+1}
	&=z_n(1-(z_n-w_n))-w_n(1+(z_n-w_n))
	=(z_n-w_n)(1-(z_n+w_n)),
\end{align*}
so
\begin{align}
f^{-(n+1)}(\{z-w=1\}) 
	&=\{(z_{n+1}-w_{n+1})=1\}  
	=\{(z_n-w_n)(1-(z_n+w_n))=1\}	\label{curvepreimage}	\\	
	&=\left\{(z-w)\prod_{j=0}^{n}(1-(z_j+w_j))=1\right\}\qquad\mbox{ and }\notag\\
	f^{-(n+1)}(\{z-w=-1\}) 
	&=\left\{(z-w)\prod_{j=0}^{n}(1-(z_j+w_j))=-1\right\}.\notag
\end{align}
Note that $f^{-(n+1)}(l^\pm)\subset f^{-(n+1)}(\{z-w=\pm1\})$. Next we will show that for any point $x\in f^{-1}(l^+)$ its preimage set $\displaystyle\cup_{n\geq0} f^{-(n+1)}(x)$ has a limit point on the fixed line $\{z=w\}$ away from the origin. The same argument works to show this for $l^-$. \\  

\textit{Claim: Any point $(z,w)\in f^{-1}(l^+)$ with $(z,w)\neq(\frac{1}{2},\frac{3}{2})$ satisfies $z+w>2$. Similarly, any point $(z,w)\in f^{-1}(l^-)$ with $(z,w)\neq(\frac{3}{2},\frac{1}{2})$ satisfies $z+w>2$.}  \\

\textit{Proof of claim.} (Refer to Figure \ref{fig1} for an illustration.) On $l^+$, $f$ is a local isomorphism since:
\begin{align*}
\left(\det(\mathrm{d}f)\right)\big|_{z-w=1}
&=\left((1-2z+w)(1+z-2w)-zw\right)\big|_{z=w+1}
=0\Rightarrow w=-1, z=0,
\end{align*}
but $0\leq w\leq\frac{1}{2}$ on $l^+$. The point $(\frac{1}{2},\frac{3}{2})$ is an end point of the curve segment $f^{-1}(l^+)$ since $f(\frac{1}{2},\frac{3}{2})=(1,0)$ is an endpoint to $l^+$ and $f$ is a local isomorphism on $l^+$. 
A quick computation shows that $(\frac{1}{2},\frac{3}{2})$ is the only intersection point of $(z-w)(1-(z+w))=1$ and $z+w=2$, hence $f^{-1}(l^+)$ must lie entirely on one side of the line $z+w=2$. 
The point $(a,b):=(3/4,\sqrt{17}/4+1/2)\in f^{-1}(l^+)$ satisfies $a+b>2$, so all of the other points $(z,w)\in f^{-1}(l^+)\setminus\{(\frac{1}{2},\frac{3}{2})\}$ must also satisfy $z+w>2$. Due to the symmetry of $f$, the same argument holds to show this for $f^{-1}(l^{-})$ by switching the roles of $z$ and $w$ in the argument.\qed \\

Let $(p,q)\in f^{-1}(l^+)\subset\{1=(z-w)(1-(z+w))\}$ and $f^2(p,q)\in I_2\subset\{0\}\times\mathcal{C}$. Let $\{(p_{-j},q_{-j})\}_{j\in\N}$ be preimages of $(p,q)$ that lie in $\R^2$ so that $(p_0,q_0):=(p,q)$ and $f(p_{-(j+1)},q_{-(j+1)})=(p_{-j},q_{-j})$. Note that: 
\begin{align*} 
f^{-n}(p,q)&\in f^{-(n+1)}(l^+)\subset\left\{1=(z-w)\prod_{j=0}^{n}(1-(z_j+w_j))\right\}
\end{align*}
Rewriting this relation using $(p,q)$ and its preimages $\{(p_{-j},q_{-j})\}_{j=0}^n$, we get:
\begin{align}
1&=(p_{-n}-q_{-n})\prod_{j=0}^{n}(1-(p_{-j}+w_{-j})).\label{pq}
\end{align}
For any $(z,w)\in\R^2$, this inequality holds (the equality holds in $\C^2$):
\begin{align}\label{z+wbound}
z_{n+1}+w_{n+1}&=z_n+w_n-(z_n-w_n)^2\leq z_n+w_n.
\end{align}
The line $z=w$ does not intersect $f^{-1}(l^+)$, so for the point $(p,q)$, inequality \eqref{z+wbound} translates to:
\begin{align}\label{pqbound}
2\leq p+q< p_{-1}+q_{-1}\leq\ldots\leq p_{-n}+q_{-n}.
\end{align}
Let $\delta:=|1-(p_{-1}+q_{-1})|>1$. Combining this bound with \eqref{pq} we get:
 \begin{align*}
1&=|p_{-n}-q_{-n}|\prod_{j=0}^{n}|1-(p_{-j}+w_{-j})|\geq |p_{-n}-q_{-n}| \delta^n
\end{align*}
Since constant $\delta>1$, $\delta^n\to\infty$ as $n\to\infty$, but $|p_{-n}-q_{-n}|\delta^n$ is bounded above by 1, hence $p_{-n}-q_{-n}\to0$. Therefore the preimages of $(p,q)$ limit to a point on $\{z=w\}$ that is not the origin since $2\leq p_{-n}+q_{-n}$ for all $n\geq0$. The same argument holds for  $f^{-1}(l^-)$. Consequently, for any point $x\in I_1\cup I_2$, its preimage set $\cup_{n\geq0}f^{-n}(x)$ has a limit point on $\{z=w\}$ that is distinct from the origin. \\  

Figure \ref{fig1} and \ref{fig2} illustrate the dynamics of $f$ in $\R^2$. Figure \ref{fig2} was drawn using Dynamics Explorer\footnote{Dynamics Explorer is a tool for exploring dynamical systems that was written by \href{http://pantherfile.uwm.edu/sboyd/www/}{Suzanne} and Brian Boyd.  It is available for download here: http://sourceforge.net/projects/detool/.}.  The characteristic direction $[1:1]$ corresponds to the line $\{z=w\}$.  In Figure \ref{fig2}, points are colored blue if they converge to the line $\{z=w\}$.  The boundary of the blue region is the collection of curves in Figure \ref{fig1} given by the preimages of $I_1$ and $I_2$.  In Figure \ref{fig2}, points are colored pale red if they do not converge to $\{z=w\}$ or to the origin.  The dark red indicates a change in behavior that should appear as a line of blue (converge to $\{z=w\}$) or yellow (converge to $(0,0)$), however the solitary curve was too thin to show up in a different color than the surrounding curves.  In particular, preimages of $\{z=w\}$ converge to $\{z=w\}$ and show up as the dark red lines that enter the blue region (they correspond to $\{z=w\},$ $f^{-1}(\{z=w\})=\{z+w=1\}\cup\{z=w\}$, etc.). Two of these lines are labeled in blue with boxes around their equations.  Similarly, the boundary of the blue region, comprised of preimages of $I_1$ and $I_2$, should be yellow as these points converge to $(0,0)$, but the curve is too thin to appear. Points on the line $z=w$ are illustrated in Figure \ref{fig2} according to their type: semi-attracting points ($0<z=w<1$) are green, parabolic points ($(0,0)$ and $(1,1)$) are black, and semi-repelling points ($z=w<0$ or $z=w>1$) are the rest of the points on $z=w$.


\begin{figure}[H]
\centering
\begin{minipage}[b]{0.47\linewidth}
\includegraphics[height=3in]{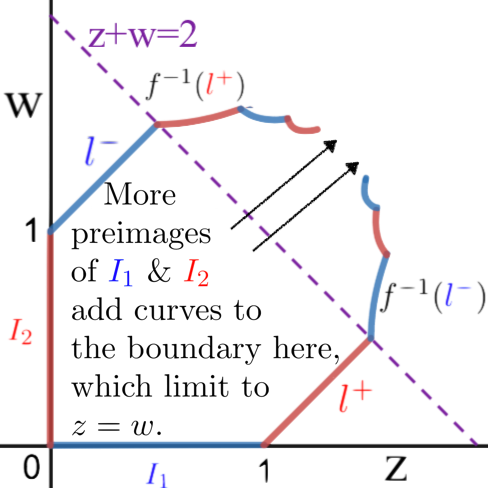}
       \caption{In $\R^2$, preimages of: \newline \textcolor{blue}{$I_1$} \& \textcolor{blue}{$l^-$} \textit{\hfill($f^{-1}(I_1)=I_1\sqcup l^-$)}; \newline \textcolor{red}{$I_2$} \& \textcolor{red}{$l^+$}  \textit{\hfill ($f^{-1}(I_2)=I_2\sqcup l^+$)}. }
\label{fig1}
\end{minipage}
\qquad
\begin{minipage}[b]{0.47\linewidth}
        \includegraphics[height=3in]{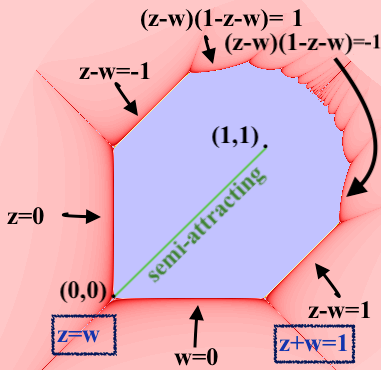}\caption{\newline Dynamical Picture in $\R^2$: Blue\newline points converge to $\{z=w\}$.} 
\label{fig2}
\end{minipage}
\end{figure}

\section{Dynamics of $f$ at the origin along $[1:1]$}\label{11}

The line $\{z=w\}$ is point-wise fixed (since $f(z,z)=(z,z)$) and corresponds to the direction $[1:1]$.  This direction is a degenerate characteristic direction of $f$ since $(p,q)(1,1)=0(1,1)=(0,0).$  It is easy to check that $[1:1]$ is an apparent characteristic direction, whereas $[1:0]$ and $[0:1]$ are both Fuchsian characteristic directions (definitions in \cite{AT,RV}).  \\
%

In $\C^2$, for a map tangent to the identity, there are many results on whether a domain of attraction exists along a given characteristic direction that depend, in part, on if the direction is apparent, Fuchsian, irregular, or dicritical (see \cite{AR,H1,H2,L1,L2,L3,V}).  
In Section \ref{nodomain} we show that there is no domain of attraction to the origin. Here we show that the dynamics near $[1:1]$ is very interesting even though there is no domain of attraction to the origin along $[1:1]$. Now we pay more rigorous attention to the interesting dynamical behavior depicted in Figure \ref{fig2}.  In particular, we prove there is an open subset of $\R^2$ whose points converge to $\{z=w\}$.  \\


For the remainder of this section, we restrict the domain of $f$ to $\R^2$.  We first show that 
\begin{equation}\label{A}
A:=\{(z,w)\in\R^2 \ | \ z>0,w>0,z+w<1\}
\end{equation}
is $f$-invariant.  By looking at where the lines $z=0,w=0$ and $z+w=1$ intersect, we see that for $(z,w)\in A$, we have $-1<z-w<1$.  For any $(z,w)\in A$, with $(z_1,w_1):=f(z,w)$:
\begin{align*}
z_1&=z(1-(z-w))>z\cdot 0=0, \\
w_1&=w(1+(z-w))>w\cdot0=0,\mbox{ and}\\
z_1+w_1&= z+w-(z-w)^2\leq z+w< 1.
\end{align*}
Hence, $A$ is $f$-invariant.  \\

Next we show that points inside $A$ converge to $\{z=w\}$.  Let $l(z,w)=(z+w,z-w):=(x,y)$ and conjugate $f$ by $l$.  Then $f$ acts on the new coordinates by:
\begin{equation}\label{xy}
(x,y)\mapsto (x_1,y_1):=l\circ f\circ l^{-1}(x,y)=\left(x-y^2,y(1-x)\right)\quad\mbox{ and }\quad[1:1]\mbox{ moves to }[1:0].
\end{equation}  
For $(z,w)\in A$, the orbit $(z_n,w_n)$ stays inside the bounded set $A$ and showing that it converges to a point on $\{z=w\}$ is equivalent to showing that $y_n\to 0$ or $|y_n|^{-1}\to\infty$. 
%
 We change coordinates again by defining $v:=y^{-1}$.  Then $f$ sends $v$ to:
\begin{align*}
v_1&=\frac{v}{1-x}=v\left(1+x+\sum_{j=2}^\infty x^{j}\right).
\end{align*}
In $A$ with $z\neq w$, we have $1>x=z+w\geq |z-w|=|v|^{-1}>0$, so:
$$|v_1|>|v| (1+x)>|v|\left(1+|v|^{-1}\right)=|v|+1\quad\mbox{ and }\quad |v_n|>|v|+n.$$
Hence, $\displaystyle\lim_{n\to\infty}|z_n-w_n|=\lim_{n\to\infty}|v_n|^{-1}=0$ and all points inside $A$ converge to $\{z=w\}$.  In addition, the rate of convergence of $z_n-w_n$ to $0$ is at least $n^{-1}$ since $|z_n-w_n|<n^{-1}$ for all $n>0$.   \\

Lastly, we show that points in $A$ do not converge to the origin.  Observe that for any $(z,w)\in A$,
\begin{align*}
\mbox{if }\quad z-w\leq 0,~ \quad&\mbox{then }\quad z_1-w_1=
(z-w)(1-z-w)\leq 0;\mbox{ and} \\
\mbox{if }\quad w-z\leq 0, ~\quad&\mbox{then }\quad w_1-z_1=
(w-z)(1-z-w)\leq 0. 
\end{align*}
Take any $(z,w)\in A$ with $z-w\leq 0$.  Then
$$\frac{z_n}{z}=\prod_{j=0}^{n-1} \frac{z_{j+1}}{z_j}=\prod_{j=0}^{n-1}  (1-(z_j-w_j))\geq 1$$
since $z-w\leq 0$ implies that $z_j-w_j\leq 0$ for all $j$.  Similarly, if $(z,w)\in A$ and $w-z\leq0$, then $\frac{w_n}{w}\geq 1$.  Therefore, for any $(z,w)\in A$, $(z_n,w_n)\not\to(0,0)$.  Hence, $A$ is $f$-invariant and points in $A$ converge to the characteristic direction $[1:1]$, but do not converge to the origin.  

\section{No domain of attraction to the origin in $\C^2$}\label{nodomain}
While some points converge to the origin under iteration, we now show that there is no (non-empty) domain of $\C^2$ whose points converge to the origin (ie.,there is no \textit{domain of attraction to the origin}).  Consider any point $(z,w)\in(\C^\times)^2$ that is near the origin; in particular, $(z,w)\in B_\epsilon:=\{(z,w)\in\C^2 \ | \ 0<|z|<\epsilon, 0<|w|<\epsilon\}$ for $\epsilon<\frac{1}{4}$.  Then:


\begin{align}
z_n\to0
&\Leftrightarrow z_n=z\prod_{j=0}^{n-1} (1-(z_j-w_j))\to 0 && \notag\\
&\Leftrightarrow \sum_{j=0}^{n-1} \log|1-(z_j-w_j)|\to-\infty&& \mbox{or }\exists k\geq0\mbox{ such that }z_k-w_k=1\notag\\
 &\Leftrightarrow \sum_{j=0}^{n-1}  (-2\re(z_j-w_j)+|z_j-w_j|^2)\to-\infty&&\mbox{or }\exists k\geq0\mbox{ such that }z_k-w_k=1\notag\\
&\Rightarrow \sum_{j=0}^{n-1}  \re(z_j-w_j)\to\infty &&\mbox{or }\exists k\geq0\mbox{ such that }z_k-w_k=1\label{eq1}\\
w_n \to0 
&\Leftrightarrow w_n=w\prod_{j=0}^{n-1}  (1+(z_j-w_j))\to 0 &&\notag\\
&\Leftrightarrow \sum_{j=0}^{n-1} \log|1+(z_j-w_j)|\to-\infty&&\mbox{or }\exists l\geq0\mbox{ such that }z_l-w_l=-1\notag\\
&\Leftrightarrow \sum_{j=0}^{n-1}  (2\re(z_j-w_j)+|z_j-w_j|^2)\to-\infty &&\mbox{or }\exists l\geq0\mbox{ such that }z_l-w_l=-1\notag\\
&\Rightarrow \sum_{j=0}^{n-1}  \re(z_j-w_j)\to-\infty &&\mbox{or }\exists l\geq0\mbox{ such that }z_l-w_l=-1.\label{eq2}
\end{align}

Since $(z,w)\in B_\epsilon$, we know that $z-w\neq\pm 1$.  The limits in \eqref{eq1} and \eqref{eq2} are mutually exclusive, so in order for $(z_n,w_n)\to(0,0)$, we must have that $z_j-w_j=1$ or $-1$ for some $j$.  \\

Suppose that $(z_n,w_n)\to(0,0)$.  Then some iterate of $(z,w)$ must first leave a neighborhood of the origin (to get $|z_j-w_j|=1$) and then converge along $\C\times\{0\}$ or $\{0\}\times\C$ since $w_{j+1}=0$ or $z_{j+1}=0$.  Consequently, the iterates of $(z,w)$ converge along the characteristic direction $[1:0]$ or $[0:1]$.  However, as we already explained in \S\ref{s1}, there is no domain of attraction along $[1:0]$ or along $[0:1]$.  Hence, there is no (non-empty) domain in $\C^2$ whose points are attracted to the origin.


\section{Dynamics of $f$ near $z=w\neq0$}\label{z=w}
As we noted earlier, there are three types of fixed points on $z=w$: parabolic, semi-attracting, and semi-repelling. So far we have focused on the parabolic fixed point $(0,0)$, but now we will analyze the dynamics in at the other parabolic fixed point, $(1,1)$, and at the semi-attracting fixed points. The point $(1,1)$ is parabolic with eigenvalues $1,-1$ under $f$ and eigenvalues $1,1$ under $f^2$, so we will investigate the dynamics near $(1,1)$ under $f^2$. Note that: 
%
%
%
\begin{align*}
f^2(z,w) 
&= \left( z(1-(z-w)(2-2z+z^2-w^2)), w(1+(z-w)(2-2w-z^2+w^2) \right)
\end{align*}
Let $k(z,w)=(z-1, w-1)$, $k^{-1}(z,w)=(z+1,w+1)$ and conjugate $f^2$ as $\hat{f}:=k\circ f^2\circ k^{-1}$ to move $(1,1)$ to the origin. Then: 
\begin{align*}
\hat{f}
&=\left( z-(z+1)(z-w)(z^2-w^2-2w) , w+(w+1)(z-w)(-z^2-2z+w^2) \right)\\
&=\left( z+2w(z-w)+\hat{p}_{>2}(z,w) , w-2z(z-w)+\hat{q}_{>2}(z,w) \right)\\
&=\Id+(\hat{p}_2(z,w),\hat{q}_2(z,w))+(\hat{p}_{>2}(z,w),\hat{q}_{>2}(z,w)),
 \end{align*}
 where $\hat{p}_2,\hat{q}_2$ are homogeneous polynomials of degree 2 and $\hat{p}_{>2},\hat{q}_{>2}$ are polynomials with terms of degree greater than 2. 
Let:
\begin{align*} 
\hat{r}(z,w)
&=z\hat{q}_2(z,w)-w\hat{p}_2(z,w)
=-2(z-w)(z^2+w^2)
=-2(z-w)(z+iw)(z-iw)
\end{align*}
The characteristic directions of $\hat{f}$ at the origin occur when $\hat{r}(z,w)=0$, so at $[1:1],[1:i],[1:-i]$.
 \begin{align*}
 (\hat{p}_2,\hat{q}_2)(1,1)&=(0,0)=0(1,1)&\Rightarrow &[1:1]\mbox{ is degenerate.}&\\
(\hat{p}_2,\hat{q}_2)(1,i)&
=-2i(1-i)(1,i)&\Rightarrow& [1:i]\mbox{ is non-degenerate.}&\\
(\hat{p}_2,\hat{q}_2)(1,-i)&
=-2i(1+i)(1,-i)&\Rightarrow& [1:-i]\mbox{ is non-degenerate.}&
\end{align*}

For the two non-degenerate characteristic direction, the directors can be computed by evaluating the following at $w=i, -i$ for $[1:i],[1:-i]$, respectively:
$$\frac{1}{\hat{p}_2(1,w)}\frac{d}{dw} \hat{r}(1,w)
=\frac{(1+w^2)-2w(1-w)}{w(1-w)}
=\frac{3w^2-2w+1}{w(1-w)}$$


The directors for $[1:i]$ and $[1:-i]$ are both $-2$, hence by \cite[Corollary 8.11]{AR} there is no domain of attraction whose points converge to the origin along $[1:i]$ or $[1:-i]$. For our original map $f$, this corresponds to having no domains of attraction to $(1,1)$ along the two non-degenerate characteristic directions originating at $(1,1)$. \\

As discussed in the introduction, we can determine what type of characteristic direction $[1:1]$ is by comparing the orders of vanishing of $\hat{p}_2(1,w)$ and $\hat{r}(1,w)$ at $w=1$.
\begin{align*}
\mu_1(1)&:=\ord_{w=1} \hat{p}_2(1,w)= \ord_{w=1}2w(1-w)=1\\
\mu_2(1)&:=\ord_{w=1} \hat{r}(1,w)= \ord_{w=1} -2(1-w)(1+iw)(1-iw)=1
\end{align*}
Since $\mu_2(1)<\mu_1(1)+1$, $[1:1]$ is a degenerate apparent characteristic direction. \\ 

To better understand the dynamics at a general fixed point $(a,a)\in\C^2$ on the fixed complex line $\{z=w\}$, we move $(a,a)$ to the origin and $\{z=w\}$ to $\{w=0\}$ via conjugation by the affine transformations: $l(z,w)=( \frac{1}{2}(z+w-2a), \frac{1}{2}(z-w))$ and $l^{-1}(z,w)=(z+w+a,z-w+a)$. Then: 
\begin{align}
\label{fa} f_a(z,w)
&:=l\circ f\circ l^{-1}(z,w)
  =\left(z-2w^2, (1-2a)w-2wz\right)\\
\notag \phi_a(z,w)&:=f_a(z,w)-\Id=\left(-2w^2,-2w(a+z) \right)       
\end{align}
Notice that $\phi_a^{-1}(0,0)=\{w=0\}$, so the origin is not isolated in its fiber. In \cite[Theorem 1.1]{H3}, Hakim showed that for semi-attracting points (so for $|1-2a|<1$ in this example) either there is a curve of fixed points or a domain of attraction at the origin. In this case, since $\phi_a$ does not have finite multiplicity at the origin, Theorem 1.1 shows that there is a curve of fixed points, but does not show a domain of attraction at the origin. In particular, $w=0$ is a curve of fixed points through the origin for $f_a$, which corresponds to the curve of fixed points $z=w$ for $f$.\\ 

\section{Adding higher degree terms to $f$}\label{hot}
In this section, we will see that adding higher degree terms to $f$ can lead to the existence of a domain of attraction to the origin along its only degenerate characteristic direction $[1:1]$.  For simplicity, we use the conjugation of $f$ given in \eqref{xy} and denoted $\tilde{f}$.  In particular, $\tilde{f}(x,y)=(x-y^2,y(1-x))$ and $[1:0]$ is its only degenerate characteristic direction.  Let $g$ be as in Theorem~\ref{thmg}:
\begin{equation}\label{g}g(x,y)=(x-y^2+ax^{r+1},~y(1-x)),\quad\mbox{ where: }\begin{cases}\mbox{(1)} & a\in\C^\times\mbox{ and }r\in\N,r\geq3\mbox{ or }\\ \mbox{(2)} & a\not\in\R_{\geq0}\mbox{ and }r=2.\end{cases}\end{equation}
We will see that $g$ and $[1:0]$ satisfy the conditions in theorem  \cite[Theorem A]{L3}, which is stated below. Before stating this theorem, we give a few relevant definitions. \\

Let $F$ be a holomorphic self-map of $\C^m$ that fixes the origin, is tangent to the identity, and is of order $k+1$. Near the origin, $F$ can be written as:\begin{equation}\label{standardf}
F(z):=z+P_{k+1}(z)+P_{k+2}(z)+\cdots,\end{equation}
where $P_j$ is a homogeneous polynomial of degree $j$ and $P_{k+1}\not\equiv O$. The following definitions extend our earlier definition of characteristic direction to polynomials besides $P_{k+1}$.

\begin{Def}Let $Q:\C^m\to\C^m$ be a homogeneous polynomial and suppose that $Q(v)=\lambda v$ for $v\in\C^m\setminus\{O\}$ and $\lambda\in\C$.  The projection of $v$ to $[v]\in\p^{m-1}(\C)$ is called a \textit{characteristic direction} of $Q$; $[v]$ is \textit{degenerate} if $\lambda=0$ and \textit{non-degenerate} if $\lambda\neq 0$.  \end{Def} 

For the following definitions, let $F$ be as in \eqref{standardf} and $[v]\in\p^{m-1}(\C)$.
\begin{Def}\label{chardir} 
$[v]$ is a \textit{characteristic direction of degree $s$} if $[v]$ is a characteristic direction of $P_{k+1},\ldots,P_{s}$, where $s\geq k+1$.  In addition, $[v]$ is \textit{non-degenerate in degree $s$} if it is a degenerate characteristic direction of $P_{k+1},\ldots,P_{s-1}$ and a non-degenerate characteristic direction of $P_{s}$, where $s>k+1$.  
\end{Def} 

\begin{Def}\label{degone} Suppose $m=2$, $[1:0]$ is a characteristic direction of $F$, and $P_j(z,w)=(p_j(z,w),q_j(z,w))$ for all $j$.  $[1:0]$ is of \textit{order one in degree $t+1$} if, for all $k+1\leq j\leq t$, $w | q_{t+1}(z,w), w^2\not|q_{t+1}(z,w)$ and $w^2 | q_j(z,w)$.  \end{Def}

Note that we can move a characteristic direction $[v]\in\p^{m-1}(\C)$ to $[1:0:\ldots:0]$ via a linear conjugation. We do that to extend Definition \ref{degone} to other directions besides $[1:0]$.

\begin{Def}$F$ is \textit{transversally attracting} in $[v]$ if $\re(\Delta)>0$, where $\Delta$ is the director of $F$.\end{Def}

\begin{rem}A director is a constant associated to a characteristic direction of $F$.  The process of finding the director of a characteristic direction $[v]$ when $[v]$ is of order $k+1$ and non-degenerate of order $k+1$ is illustrated in \eqref{dir1} and \eqref{dir2}.  For a more general definition of director, which extends beyond this case see \cite{L3}.   \end{rem}

\begin{thmB}\label{thmB}
Let $g$ be a germ of a holomorphic self-map of $\C^2$ that is tangent to the identity at the fixed point $\0$, is of order $k+1$, and has characteristic direction $[v]\in\p^1(\C)$.  Assume $[v]$ is:
\begin{enumerate}
\item a characteristic direction of degree $s\leq\infty$;
\item non-degenerate of degree $r+1$, where $k<r<s$; and
\item of order one in degree $t+1$, where $k\leq t\leq r$.
\end{enumerate}
If $[v]$ is transversally attracting and $s>r+t-k$, then there exists a domain of attraction whose points, under iteration by $g$, converge to $\0$ along $[v]$. 
\end{thmB} 

In particular, $g$ with $[v]=[1:0]$ satisfies the conditions of this theorem since $1=k=t<r<s=\infty$ and $[v]$ is \textit{transversally attracting} by assumption (1) or (2) in \eqref{g}.  Hence, $g$ has a domain of attraction to the origin along $[1:0]$.  Note that $g$ and $[1:0]$ would continue to satisfy the conditions of \cite[Theorem A]{L3} if we added many different types of higher degree terms to $g$.   In particular, we could add $\eta$ to $g$, where $\eta(x,y)=\left(y\BO((x,y)^2)+\BO(x^{r+2}),~y\BO((x,y)^2)+\BO(x^{r+2})\right).$\\

Estimates for the rates of convergence of $g^n(x,y)=(x_n,y_n)$ within the domain of attraction are given in \cite[Proposition 4.2]{L3}.  In particular, for $t=k=1<r$ and $s=\infty$:
$$x_n\lesssim n^{-\frac{1}{r}}\qquad\mbox{ and }\qquad
y_n\lesssim e^{-\re\beta\frac{r}{r-1}n^{\frac{r-1}{r}}},$$
where $\beta:=(-ar)^{-\frac{1}{r}}$ and the $\frac{1}{r}$-th root is chosen so that $\re\beta>0$, which we can do when $r>2$ or $r=2$ and $a\not\in\R_{\geq0}$.  If we allow $r$ to become arbitrarily large, the growth rates approach:
\begin{equation}\label{rinfty}
x_n\lesssim 1 \qquad\mbox{ and }\qquad
y_n\lesssim e^{-n \re\beta}.\end{equation}
The domain of attraction for $g$ given in \cite{L3} is $g$-invariant for arbitrarily large $r$, but it is no longer invariant if $r=\infty$, so we cannot quite extrapolate that the growth rates in \eqref{rinfty} hold when $r=\infty$.  However, the growth rates help support the following conjecture for $f$ since, near the origin, the behavior of $g$ approaches the behavior of $\tilde{f}$ as $r$ goes to infinity.

\begin{conj}\label{conj1}There is a non-empty, $f$-invariant domain in $\C^2$ whose points, under iteration, converge in $\p^1$ to its degenerate characteristic direction $[1:1]$.\end{conj}

The author can further support this conjecture from computations done using Dynamics Explorer.  In particular, it appears that there is a domain with the origin in its boundary and $\re z,\re w>0$ (equivalently, $\re (x+y),\re(x-y)>0$) whose points converge to $\{z=w\}$ (equivalently, $\{y=0\}$). \\ 

An interesting borderline case is when $r=2$ in \eqref{g}.  Then the conditions of \cite[Theorem A]{L3} are not automatically satisfied.  
In particular, for this case, the conditions of the theorem are not satisfied if and only if $a\in\R_{\geq 0}$.  We use $h$ to represent this  special case (assume $a\neq0$ so $h\neq \tilde{f}$):
\begin{equation}\label{h}h(x,y)=(x-y^2+ax^{3},~y(1-x)),\quad\mbox{ for any }a\in\R_{> 0}.\end{equation}
Theorem~\ref{thmh} is about the dynamics of $h$ and it is proven later on in this section. \\

Before proving Theorem~\ref{thmh}, we compare $g,h$ and $f$ graphically.  Converting $g$ and $h$ to $(z,w)$-coordinates we get:
$$\tilde{g}(z,w)=\left(z(1-(z-w))+\frac{a}{2}(z+w)^{r+1},~w(1+(z-w))+\frac{a}{2}(z+w)^{r+1}\right).$$
Notice that: $\tilde{g}=f$ when $a=0$, $\tilde{g}$ corresponds to $h$ when $r=2$ and $a\in\R_{>0}$, and $\tilde{g}$ corresponds to $g$ otherwise.  Since $\tilde{g}$ arises from $f$ by adding higher degree terms only, the types of characteristic directions for both maps are the same.  In particular, $f$ and $\tilde{g}$ have $[1:0],[0:1]$ as non-attracting, non-degenerate characteristic directions and $[1:1]$ as a degenerate characteristic direction.  Consequently, the only hope for $\tilde{g}$ to have a domain of attraction to the origin along a direction is to have it along $[1:1]$. Figures \eqref{-fig}-\eqref{+fig} are dynamical pictures in $\R^2$ of $\tilde{g}$ with $r=2$ and $a=\pm0.1,0$.  

\begin{figure}[H]
\centering
\begin{minipage}[b]{0.3\linewidth}
        \includegraphics[height=1.8in]{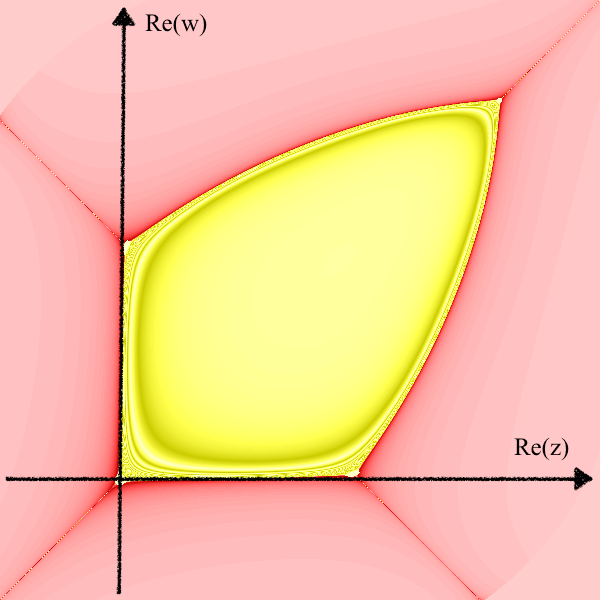}\caption{
        \newline$a=-0.1\Rightarrow g$
        }
\label{-fig}
\end{minipage}
\quad
\begin{minipage}[b]{0.31\linewidth}
        \includegraphics[height=1.8in]{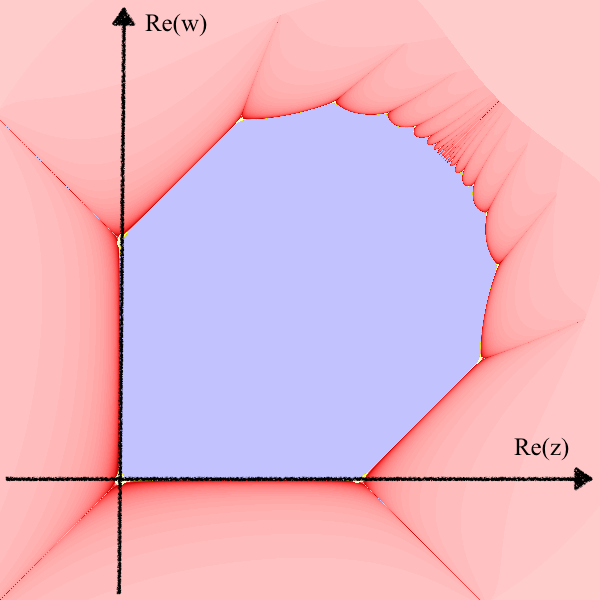}
        \caption{\newline
        $a=0\Rightarrow f$
        }\label{0fig}
\end{minipage}
\quad
\begin{minipage}[b]{0.3\linewidth}
        \includegraphics[height=1.8in]{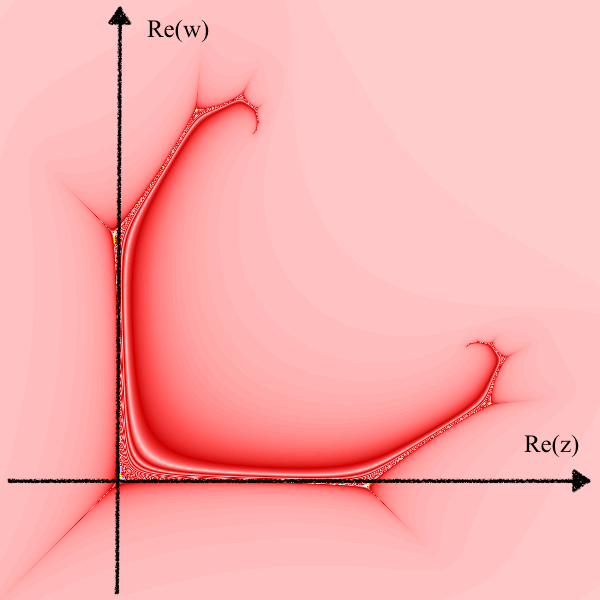}\caption{
        \newline$a=0.1\Rightarrow h$ 
        }
\label{+fig}
\end{minipage}\vspace{.1in}
These three figures show the dynamics of $\tilde{g}$ in the window $[-0.5,2]^2\subset\R^2$ with $r=2$.  Points are:
 \begin{itemize} 
\item yellow if they converge to $\0$,
\item blue if they converge to $[1:1]$, but not to $\0$, and
\item red if they leave a neighborhood of $\0$ (in particular, $|(z_j,w_j)|>5$ for some $j$). 
\end{itemize}
 \end{figure}\vspace{-.2in}
The different shades of a given color represent how quickly a point converges or diverges.  Note that some curves that converge to $[1:1]$ appear red because nearby points that do not lie on those curves diverge from $[1:1]$ and $\0$; a single curve is too thin to appear blue on its own.  \\  

Notice that $\tilde{g}$ sends $\{z=w\}$ to $\{z_1=w_1=z(1+a(2z)^r)\}$, so $\{z=w\}$ is $\tilde{g}$-invariant, but not pointwise fixed (unless $a=0$ and so $f=\tilde{g}$).  Restricting to $\R^2$ and $r=2$, it is clear that points on $\{z=w\}$ converge to the origin for sufficiently small $|z|$ exactly when $a<0$.  Expanding to $\{z=w\}\subset\C^2$ and $r\geq2$, there is an open subset of $\{z=w\}$ that converges to the origin as long as $a\neq0$ by the Leau-Fatou Flower Theorem \cite{A3}.  \\

Now we focus on the case when $r=2$ and $a\in\R_{>0}$, as in \eqref{h} and Figure \ref{+fig}.

\begin{thmh}\label{thmH}Let $h(x,y)=(x-y^2+ax^3,y-xy)$.  If $a\in\R_{>0}$, then $h$ has no domain of attraction to the origin along any direction.\end{thmh}

\begin{proof}First of all, by \cite[Proposition 1.3]{H1}, if an orbit of a point converges to the origin along a direction, that direction must be a characteristic direction.  Since $h$, up through degree 2, is conjugate to $f$ from Theorem~\ref{mainthm} it has the same types of characteristic directions as $f$.  In particular, $[1:1]$ and $[1:-1]$ are both non-degenerate and non-attracting characteristic directions (for each, the real part of its director is negative).  Hence, $h$ has no domain of attraction along either of those directions.  The remaining characteristic direction, $[1:0]$, is degenerate.   \\

Suppose, for contradiction, that there exists an $\Omega\subset\C^2$ that is a domain of attraction to the origin along $[1:0]$.  Then $\Omega$ is an open set with $\0\in\partial\Omega$ and for any $(x,y)\in\Omega$, the iterates $f^n(x,y)=(x_n,y_n)\to(0,0)$ and $\frac{y_n}{x_n}\to0$.  If $y=0$, then $h(x,y)=\left(x(1+ax^2),0\right)$ and the Leau-Fatou Flower Theorem tells us that there are domains in $\C\times\{0\}$ whose points are attracted to the origin.  That gives us open subsets of $\C\times\{0\}$ that converge to the origin (trivially) along $[1:0]$.  Let $u=\frac{y}{x}$ and $v=\frac{y}{x^2}=\frac{u}{x}$.  Then:
\begin{align} \label{xuv}
x_1&
	=x(1+x(ax-u^2))&&=x(1+x^2(a-xv^2))& \\
u_1&=u\frac{1-x}{1+x(ax-u^2)}
	&&=u\left(1-x(1+(ax-u^2))+x^2\BO(x,u^2)\right)&\notag\\
	&&&=u\left(1-x+\BO(x^2,xu^2)\right)&\notag\\ 
v_1&=v\frac{1-x}{1+x(ax-u^2)(2+x(ax-u^2))}
	&& =v\left(1-x\left(1+2(ax-u^2)\right)+x^2\BO(x,u^2)\right)&\notag\\
		&&&=v\left(1-x+\BO(x^2,xu^2)\right)& \notag
\end{align}
where $(x,y)\in\Omega$ and we assume that $|x|,|u|\ll 1$ so that we can use the geometric series to get the equalities on the right.  We can use that $|x|$ and $|u|$ are small since we are working near the origin and $(x,y)\in\Omega$ so $x_n\to0$ and $u_n\to0$.  The $n$th-iterates of $x,u,$ and $v$ are:
\begin{align}\label{xuvn}
x_n&=x\prod_{j=0}^{n-1} (1+x_j(ax_j-u_j^2))=x\prod_{j=0}^{n-1} (1+x_j^2(a-x_jv_j^2)) \\
u_n&=u\prod_{j=0}^{n-1} \left(~1-x_j\left(1+(ax_j-u_j^2)+x_j\BO(x_j,u_j^2)\right)~\right)
=u\prod_{j=0}^{n-1} \left(~1-x_j+\BO(x_j^2,x_ju_j^2)~\right)
 \notag\\
v_n&=v\prod_{j=0}^{n-1} \left(~1-x_j\left(1+2(ax_j-u_j^2)+x_j\BO(x_j,u_j^2)\right)~\right)=v\prod_{j=0}^{n-1} \left(~1-x_j+\BO(x_j^2,x_ju_j^2)~\right)
,\notag 
\end{align}
where $(x,y)\in\Omega$ and we use $|x_j|,|u_j|\ll1$ as we did, with $j=0$, to get \eqref{xuv}. 
If $y\neq0$ and $y_n\to0$ as $n\to\infty$, then $\log|y_n|\to-\infty$ implies:
\begin{equation}\label{yn}
\sum_{j=0}^{n-1}\log|1-x_j|\to-\infty\Rightarrow\sum_{j=0}^{n-1}\log(1-2\re x_j+|x_j|^2)\to-\infty
\Rightarrow \sum_{j=0}^{n-1} \re x_j\to\infty.\end{equation}
Notice that the final expressions for $u_n$ and $v_n$ in \eqref{xuvn} appear the same.  For $|x|$ and $|u|$ sufficiently small, the big-$\BO$ in $u_n$ and $v_n$ effectively act the same and so $u_n\to0$ implies $v_n\to0$. \\ 

Now we show $\Omega$ cannot exist when $a>0$.  
Let $t=\frac{1}{x}$.  
Using the expression for $x_1$ in \eqref{xuv}, we get:
\begin{align*}
t_1&= \frac{t}{1+\frac{1}{t^2}(a-\frac{v^2}{t})}
	=t-\frac{a}{t}+\BO\left(\frac{v^2}{t^2},\frac{1}{t^3}\right) \\	
t_n &=t-a\sum_{j=0}^{n-1} \frac{1}{t_j}\left(1+\BO\left(\frac{v_j^2}{t_j},\frac{1}{t_j^2},\right)\right)	
\end{align*}
Since $|v_j|$ and $|t_j|^{-1}=|x_j|$ are small and go to 0 as $j\to\infty$, the big-$\BO$ terms in $t_n$ are much smaller than 1.  Hence,
$$\re t_n\approx \re t-a\sum_{j=0}^{n-1} \re\left(\frac{1}{t_j}\right)=\re t-a\sum_{j=0}^{n-1} \re x_j\to -\infty,$$
where we get the limit on the right from \eqref{yn} and our assumption that $a>0$.  This implies that for some $N>1$, which can depend on $(x,y)$, $\re t_n<0$ for all $n>N$.  Note that:
$$\re x_j=\re\left(\frac{1}{t_j}\right)=\frac{1}{|t_j|^2} \re t_j=|x_j|^2\re t_j.$$
Combining this and that $y_n\to0$ implies $\sum\re x_j\to\infty$ from \eqref{yn}:
$$\sum_{j=1}^{\infty}\re x_j
=\sum_{j=1}^{N}\re x_j+\sum_{j=N+1}^{\infty}|x_j|^2 \re t_j
=(\mbox{finite number})+(\mbox{sum of negative numbers})\not\to+\infty.$$
This is a contradiction.  
Hence, there is no domain of attraction to the origin along $[1:0]$.
\end{proof}

\begin{conj}\label{conj2}For $h$ is as in Theorem~\ref{thmh}, there is no domain (in $\C^2$) whose points converge to $\0$.\end{conj}

We showed in Theorem~\ref{thmh} that if such a domain exists, none of its points can converge along a direction.  For the example in Figure~\ref{+fig}, we see that if such a domain exists, it should not intersect $\R^2\setminus\{\0\}$.  The author can further support this conjecture from numerous computations done using Dynamics Explorer, where $a>0$ was varied and $z,w$ were also considered in $\C$ (not just $\R$).

\section{Summary}\label{summary}
The main example in this paper, 
$$f(z,w)=(z(1-(z-w)),w(1+(z-w))),$$
has no domain in $\C^2$ whose points are attracted to the origin.  The characteristic direction $[1:1]$, which corresponds to $\{z=w\}$, attracts points in an open set $A\subset\R^2$ to itself, but not to the origin.  The author, in Conjecture \ref{conj1}, hypothesizes that there  exists a domain in $\C^2$ that is attracted to $\{z=w\}$.  This analysis of $f$, which is summarized in Theorem \ref{mainthm}, provides a fairly detailed description of the dynamics of $f$ in a full neighborhood of the origin.  By adding higher degree terms to $f$, as we did in \S\ref{hot}, we can significantly alter the behavior of points along the degenerate characteristic direction.  Sometimes, but not always, the addition of higher degree terms can lead to stronger convergence than in $f$; in particular, the existence of a domain of attraction.  \\

The dynamical behavior of the main example in $\R^2$ is summarized in the following Figure \ref{fig3}.  The region $A$ is the triangle bounded by the axes and the purple line.  The dynamics in $A$ is of particular interest as it is the first example, as far as the author knows, of a set whose points: (1) can remain arbitrarily close to the origin under iteration, (2) are not all fixed, (3) converge to a characteristic direction, and (4) do not converge to the origin.

\begin{figure}[H]\caption{}\vspace{.1in}
\centering
\begin{minipage}[b]{0.33\linewidth}
        \includegraphics[width=\linewidth]{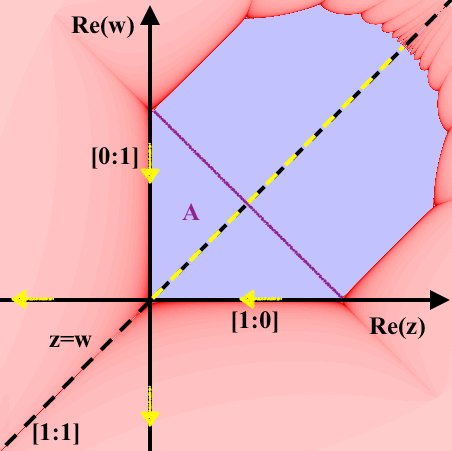}\label{fig3}
\end{minipage}
\quad
\begin{minipage}[b]{0.62\linewidth}The dynamics of $f$ in $\R^2$ are illustrated in this figure.  Yellow arrows indicate direction of movement.  For $(z,w)\in\R^2$:  \newline
 $\mbox{ }\bullet$ if $z=0$ and $0<w<1$, then $(z_n,w_n)\to(0,0)$ along $[0:1]$; 
\newline
       $\mbox{ }\bullet$  if $w=0$ and $0<z<1$, then $(z_n,w_n)\to(0,0)$ along $[1:0]$; 
\newline
	 $\mbox{ }\bullet$ if $(z,w)\in A$ (or in blue region), then $(z_n,w_n)\not\to(0,0)$ and\newline
	  $\mbox{}$\quad $[z_n:w_n]\to[1:1]$; and
	\newline
 $\mbox{ }\bullet$ if $(z,w)$ in red region, then $(z_n,w_n)\not\to(0,0)$ \newline $\mbox{}$\quad and $[z_n:w_n]\not\to[1:1]$.     
 \newline
	  $\mbox{ }\bullet$ if $x\in([0,1]\times\{0\})\cup(\{0\}\times[0,1])$, then set $\cup_{n\geq0} f^{-n}(x)\cap\R^2$ \\ $\mbox{}$\quad has a limit point on $z=w$ distinct from the origin.	\newline
	\textit{This picture was drawn using a profile written for the program Dynamics Explorer.} \newline
	The domain pictured is: $[-0.8,1.6]\times[-0.8,1.6]$.  
\end{minipage}
\end{figure}

\end{document}